\documentclass[12pt]{amsart}

\usepackage{color,graphicx,enumerate,wrapfig,amssymb}
\usepackage{subfigure}
\usepackage{eucal}

\newcommand{\R}{{\mathbb R}}
\newcommand{\Z}{{\mathbb Z}}

\newcommand{\e}{\varepsilon}

\newcommand{\To}{\rightarrow}

\newtheorem{theorem}{Theorem}[section]

\newtheorem{definition}{Definition}[section]

\newtheorem{lem}[theorem]{Lemma}

\newtheorem{remark}[theorem]{Remark}

\numberwithin{equation}{section}

\title [Homogenization of Dirichlet problem in polygonal domains]{Applications of Fourier analysis in homogenization of Dirichlet problem III: \\ Polygonal Domains}

\author{Hayk Aleksanyan}
\address{School of Mathematics, The University of Edinburgh, JCMB The King's Buildings, Mayfield Road, Edinburgh EH9 3JZ}
\email{h.aleksanyan@sms.ed.ac.uk}
\thanks{H. Aleksanyan thanks G\"{o}ran Gustafsson Foundation for visiting appointment to KTH}

\author[Henrik Shahgholian ]{Henrik Shahgholian}
\address{Department of Mathematics, KTH Royal Institute of Technology,
  100~44  Stockholm, Sweden}
\email{henriksh@kth.se}

\author{Per Sj\"{o}lin}
\address{Department of Mathematics, KTH Royal Institute of Technology,
  100~44  Stockholm, Sweden}
\email{pers@math.kth.se}
\keywords{Homogenization, Dirichlet Problem, Polygonal Domain, Fourier Analysis}
\begin{document}

    \begin{abstract}
In this paper we prove convergence results for the homogenization of the Dirichlet problem with rapidly oscillating boundary data in convex polygonal domains.
Our analysis is based on integral representation of solutions. Under a certain Diophantine condition on the boundary of the domain and smooth coefficients we prove pointwise, as well as $L^p$ convergence results. For larger exponents $p$ we prove that the $L^p$ convergence rate is close to optimal. We shall also suggest several directions of possible generalization of the result in this paper.
    \end{abstract}

\maketitle

\section{Introduction and Main Results}\label{sec-intro}
Elliptic boundary value problems with rapidly oscillating boundary data as well as oscillating coefficients  has been much in focus lately, due to its importance for higher order approximation in homogenization theory. Higher order approximation gives rise to the so-called boundary-layer phenomena, which roughly states that the solutions to elliptic problems with oscillating coefficients and boundary data should have concentration near the boundary of the domain with no periodic character. We refer the readers to \cite{AA99} for some background, and examples of applications where oscillating data plays central role.

For a smooth and uniformly convex domains in $\R^d$, $(d\geq 2)$ in a recent work \cite{GM1}, D. G\'{e}rard-Varet,  and  N. Masmoudi, proved convergence rate of order any $\alpha <(d-1)/(3d+5)$ in $L^2$ for solutions to elliptic system of divergence type, with periodically oscillating coefficients and boundary data.
This is one of the few results where the speed for such type of homogenization problem is established.
In the same setting, for homogenization of non-oscillating operators and oscillating boundary data in dimensions greater than two, the current authors showed a power convergence rate of order $1/p$ in $L^p$, for all $1\leq p<\infty$. They also proved that the rate $1/p$ can not be improved; see \cite{ASS2}.

A wider range of treatments of the problem, but with no particular speed of convergence, can be found in recent works: \cite{BM}, \cite{CK}, \cite{CKL}, \cite{Feld}, \cite{LS}, \cite{LY}.

In case, when the operator is fixed, and only the boundary data is oscillating, the convergence result was proved in \cite{LS} for some general class of domains. For elliptic systems of divergence type, the current authors  found partial convergence rate for the pointwise convergence, and an optimal rate of the convergence in $L^p$ norm in dimensions greater than three, when the domain in question is strictly convex and smooth; see \cite{ASS}, \cite{ASS2}. In this paper, we continue our program of studying the problem of homogenization of the boundary data with fixed operator.
Here we shall consider convex polygonal domains, see also \cite{GM2}.

To fix the ideas, let $D$ be a bounded convex polygonal domain in $\R^d$ $(d\geq 2)$, that is a convex domain bounded by some number of hyperplanes
\begin{equation}\label{polygon}
D=\bigcap\limits_{j=1}^N \{x\in \R^d: \ \nu_j \cdot x>c_j   \},
\end{equation}
where $c_j \in \R$ and $\nu_j \in \mathbb{S}^{d-1}$. Denote by $\Gamma$ the boundary of $D$. Let also $A(y)=(A^{ i j }(y))$, $1\leq i, j \leq d$, be an $\R^{ d^2}$-valued function defined on $\R^d$, and $g$ be a complex valued function defined on $\mathbb{T}^d$-the unit torus in $\R^d$. We study asymptotic behavior of solutions to the following problem:
\begin{equation}\label{problem-formulation}
\begin{cases}  \mathcal{L} u_{\e}(x) =0,&\text{ in $D$}, \\
u_{\e}(x)=g(x/\e)
,&\text{ on $\Gamma$, } \end{cases}
\end{equation}
where $\e>0$ is a small parameter, and using the summation convention of repeated indices the operator $\mathcal{L}$ is defined as
$$
\mathcal{L}u:=- \frac{\partial}{\partial x_i} \left[ A^{ i j } ( x )  \frac{\partial u}{\partial x_j }   \right]=-\mathrm{div} \left[ A(x )\nabla u  \right].
$$

For $(\ref{problem-formulation})$ we consider the corresponding homogenized problem
\begin{equation}\label{problem-formulation-homogen}
\begin{cases} \mathcal{L}  u_0(x)=0,&\text{ in $D$}, \\
u_0(x)=\overline{g}
,&\text{ on $\Gamma$, } \end{cases}
\end{equation}
where $\overline{g}=\int\limits_{\mathbb{T}^d} g(y)dy$.

\subsection{Standing Assumptions}\label{ass} We make the following assumptions:

\begin{itemize}
\item[(i)] (Periodicity) The boundary function $g$ is 1-periodic:
$$
g(x+h)=g(x), \   \forall x \in \R^d, \ \forall h\in \Z^d.
$$
\item[(ii)] (Ellipticity) There exists a constant $c>0$ such that
$$
  c^{-1} \xi_i  \xi_i \leq A^{i  j}(x) \xi_i \xi_j  \leq c \xi_i  \xi_i , \ \forall x \in \R^d, \ \forall \xi \in \R^{d}.
$$
\item[(iii)] (Convexity) $D$ is convex and for any bounding hyperplane of $D$ its normal vector is Diophantine in a sense of Definition \ref{Diophantine-directions} below.
\item[(iv)] For the convex polygonal domain $D$ choose $\alpha_* >0$ so that $\pi / (1+\alpha_*)$ be the maximal angle between any two adjacent faces of $D$.
\item[(v)] (Smoothness) The boundary value $g$ and all elements of $A$ are sufficiently smooth.
\end{itemize}

\noindent The following are the main results of the paper.

\begin{theorem} (\textbf{Pointwise convergence})\label{Thm-Pointwise}
Retain the standing assumptions in Section \ref{ass}, and if $\alpha_*>1$ set $\beta=1$, otherwise, let $0<\beta <\alpha_*$ be any number. Then for each $\delta>0$ small there exists a constant $C $ depending on $\delta$, $\beta$, $D$, $\mathcal{L}$, but independent of $\e>0$, such that for all $x\in D$ one has
$$
|u_\e (x) - u_0(x) | \leq C  \left(  \frac{\e^{\beta}}{d(x)^{\beta+\delta}} \right)^{\frac{d-1}{d-1+\beta}},
$$
where $d(x)$ denotes the distance of $x$ to the boundary of $D$.
\end{theorem}

\noindent Using this we will have the following result.

\begin{theorem} (\textbf{$L^p$ convergence}) \label{Thm-Lp}
Retain the standing assumptions in Section \ref{ass}, and set $ \gamma=\frac{ (d-1) \min\{1, \alpha_* \} }{d-1+ \min \{ 1, \alpha_* \} }   $. Then for each $1\leq p<\infty$, and $\delta>0$ there exists a constant $C$ depending on $p$, $D$, $\mathcal{L}$, $\delta$ but independent of $\e>0$ such that
$$
|| u_\e - u_0||_{L^p(D)} \leq C \e^{ \min\{\gamma, \frac 1p \} -\delta }.
$$
\end{theorem}

\noindent The next result shows that for larger exponents $p$ the $L^p$ convergence rate is close to optimal.
\begin{theorem}(\textbf{Optimality})\label{Thm-Optimality}
Under the same conditions and notation of Theorem \ref{Thm-Pointwise} for each $1\leq p <\infty$ there exists a constant $C$ depending on $p$, $D$, $\mathcal{L}$, but independent of $\e$, such that
$$
|| u_\e - u_0 ||_{L^p(D)} \geq C \e^{\frac 1p  } || g - \overline{g} ||_{L^\infty ( \mathbb{T}^d )}.
$$
\end{theorem}

\begin{remark}
  In Section \ref{horison} we  suggest several directions of possible generalization of  present results.
\end{remark}

\subsection{Preliminaries} We start with some auxiliary results. In the sequel we will denote by $C$ an absolute constant which may vary from formula to formula. For $x\in \R^d$ and $r>0$ we set by $B(x,r)$, or $B_r(x)$ an open ball of radius $r$ centered at $x$. If ambiguity does not arise, for a vector $x\in \R^d$ we will write $|x|$ to denote its standard Euclidean norm.

\begin{definition}\label{Diophantine-directions}
A vector $\nu =(\nu_1,...,  \nu_d)\in \R^d$ is called Diophantine if there exists $0<\tau(\nu)<\infty$ and $C>0$ such that
$$
|  m \cdot \nu | > \frac{C}{ |m|^{\tau(\nu)} },
$$
for all $m=(m_1,...,m_d) \in \Z^d\setminus \{0\}$, where $ m \cdot \nu$ is the usual scalar product and $|m|=|m_1|+...+|m_d|$.
We denote the set of such vectors by $\Omega(\tau,C)$.
\end{definition}

It is well known and easy to see that for any $\tau> d-1$ the set $\bigcup\limits_{C>0} \Omega(\tau,C)$ has full measure in each ball of $\R^d$. This shows that the Diophantine condition, as stated in (iii) of Standing Assumptions, is generic for all polygonal domains.

\begin{lem}\label{Lem-Diophantine-Decay} Let $m\in \Z^d$ be non zero, and assume that $m_k\neq 0$, for some $1\leq k \leq d$. For a vector $\nu=(\nu_1,\nu_2,...,\nu_d) \in \Omega(\tau, c_0)$ consider $ \Pi =\{ x\in \R^d: \ \nu \cdot x=c,   \   x_j \in [a_j,b_j], j=1,2,...,d, j\neq k   \}$, and for $\lambda>1$ set
$$
\mathcal{I}_{\lambda}:= \int\limits_{\Pi} e^{2 \pi i \lambda m\cdot y } d \sigma(y).
$$
Then for all $\lambda>1$ one has
$$
|\mathcal{I}_{\lambda}|\leq  C \lambda^{-(d-1)} || m ||^{(d-1) \tau },
$$
where the constant $C$ depends on $\nu$ and dimension $d$ only.
\end{lem}

\begin{proof}
Without loss of generality we will assume that $k=d$, that is $m_d\neq 0$. Since $\nu$ is Diophantine, all its components are non zero. In the domain of integration we have
$\nu_1 y_1 +...+\nu_{d-1} y_{d-1}+ \nu_d y_d=c$, hence
$$
y_d=\frac{c}{\nu_d}-\frac{1}{\nu_d}(\nu_1 y_1 +\nu_2 y_2+...+\nu_{d-1} y_{d-1}),
$$
and substituting this in the integral we obtain
\begin{equation}\label{Dioph-integral}
\mathcal{I}_{\lambda}=C \prod\limits_{j=1}^{d-1} \int\limits_{a_j}^{b_j} \exp \left[2\pi i \lambda \left(m_j - m_d \frac{\nu_j}{\nu_d} \right) y_j \right] d y_j.
\end{equation}
From the Diophantine condition and the fact that $m_d\neq 0$ we have
\begin{equation}\label{Dioph-vector}
|m_j - m_d \frac{\nu_j}{\nu_d}| =\frac{1}{|\nu_d|} |m_j \nu_d -m_d \nu_j | \geq \frac{C_\nu}{|\nu_d|} \frac{1}{(|m_j|+|m_d|)^\tau},
\end{equation}
for all $j=1,2,...,d-1$. We now compute each of the integrals in (\ref{Dioph-integral}), and applying (\ref{Dioph-vector}) we get the desired estimate, finishing the proof.
\end{proof}

We now introduce some notation that will be used in the sequel. Let $D$ be given as in (\ref{polygon}). We say that $\Pi \subset \partial D$ is a ($(d-1)$-dimensional) \emph{face} of the polygon $D$ if for some $1\leq k \leq N$ one has
\begin{equation}\label{face-of-polygon}
\Pi=\{ x\in \R^d: \nu_k \cdot x = c_k \} \cap \bigcap\limits_{j=1, j \neq k }^N \{x\in \R^d: \nu_j \cdot x >c_j \},
\end{equation}
i.e., $\Pi$ is just one of the flat portions of $\partial D$. For a given face $\Pi$, and a number $\rho>0$ consider a strip of width  $\rho$ near the $(d-2)$-dimensional boundary of $\Pi$, and denote it by
\begin{equation}\label{small strip in the face}
\Pi_\rho = \{y\in \Pi: \mathrm{dist}( y, \partial \Pi )\leq \rho  \}.
\end{equation}

For $1\leq k \leq d$ set $\pi_k$ to be the projection operator in the $k$-th direction, namely
$$
\pi_k(x)=(x_1,...,x_{k-1},0,x_{k+1},...,x_d), \text{ where } x\in \R^d.
$$

We also set $\mathcal{H}^j$ for the $j$-dimensional Hausdorff measure, for $0\leq j \leq d$.

\begin{lem}\label{Lem-Partition by lattice}
Let $D$ be a polygon as defined in $(\ref{polygon})$, and $\Pi \subset \{ x\in \R^d: \nu \cdot x = c \} $ be a face of $D$. Fix $1\leq k \leq d$, then for any small number $\rho>0$ there exist a finite number of measurable sets $\Gamma_j \subset \Pi$, $j=1,2,...,M$ with disjoint $d-1$-dimensional interiors, and a measurable set $E\subset \Pi$ such that
\begin{itemize}
\item[(i)] $E\subset \Pi_{c_0 \rho}$, for some constant $c_0$ depending on $\Pi$ and dimension $d$, but independent of $\rho$,
\item[(ii)] $\Pi\setminus E=\bigcup\limits_{j=1}^M \Gamma_j$, and $\pi_k(\Gamma_j)$ is a $(d-1)$-dimensional cube of side length $\rho$ with vertices in the lattice $\pi_k( \rho \Z^d )$, for $j=1,2,...,M$.
\item[(iii)]  for $j=1,2,... ,M$ one has $ \mathcal{H}^{d-1}(\Gamma_j ) \approx \rho^{d-1}$, and $ \mathrm{diam} (\Gamma_j) \approx   \rho  $, where constants in the equivalence depend on $\Pi$ and dimension $d$, but are independent of $\rho$.
\end{itemize}
\end{lem}

\begin{proof}
We first construct the projections of the required sets in the projection of $\Pi$, and then lift it up to $\Pi$. To have a control on the lifted sets we need some control on the projection $\pi_k$. For any $x,y\in \Pi$ one has
\begin{equation}\label{proj-distortion-ineq}
\frac{| \nu_k |}{ || \nu|| } || x- y || \leq || \pi_k (x) -\pi_k(y) || \leq ||x-y||,
\end{equation}
where $\nu=(\nu_1,..., \nu_d)$ is the unit outward normal vector of $\Pi$. The second inequality is obvious, for the first one observe that if $x\in \Pi$ then $x_k=\frac{c}{\nu_k} - \frac{1}{\nu_k} \sum\limits_{ i\neq k} \nu_i x_i $, from which we get
\begin{multline*}
|| x-y ||^2 = \sum\limits_{ i\neq k } ( x_i -y_i )^2 +\frac{1}{\nu_k^2} \left(  \sum\limits_{i\neq k} \nu_i (x_i - y_i) \right)^2 = || \pi_k(x) -\pi_k(y) ||^2 + \\
 \frac{1}{\nu_k^2} \left(  \sum\limits_{i\neq k} \nu_i ( x_i - y_i) \right)^2 \leq || \pi_k(x) - \pi_k(y) || ^2 + \frac{1}{\nu_k^2}  \sum\limits_{i\neq k} \nu_i^2    \sum\limits_{i\neq k} ( x_i -y_i )^2 = \\ || \pi_k(x) - \pi_k(y) ||^2 + || \pi_k(x) - \pi_k(y) ||^2  \frac{1}{\nu_k^2 } \sum\limits_{i\neq k} \nu_i^2,
\end{multline*}
and the first inequality in (\ref{proj-distortion-ineq}) follows. Notice that the first inequality shows that $\pi_k : \Pi \To \pi_k(\Pi) $ is a bijection.

Now consider the projection $\pi_k( \Pi )$, and let $\mathcal{C}=\{ \mathcal{C}_j \}_{j=1}^M$ be a maximal family of lattice cubes of size $\rho$ and vertices from $\pi_k( \rho \Z^d )$, such that $\mathcal{C}_j \subset \pi_k(\Pi)$. Set $\mathcal{S}=\{ x\in \pi_k( \Pi ): \mathrm{dist} (x, \partial \pi_k(\Pi))\leq   2\sqrt{d-1} \rho  \}$-a strip near the $(d-2)$-dimensional boundary of $\pi_k(\Pi)$. Since the diameter of each $(d-1)$-dimensional cube of size $\rho$ is $\sqrt{d-1} \rho$, it is clear that the set $\pi_k(\Pi) \setminus \mathcal{S}$ is entirely covered by the family of cubes $\mathcal{C}$. Now set $E_0=\pi_k(\Pi) \setminus \bigcup\limits_{\mathcal{C}_j \in \mathcal{C}} \mathcal{C}_j $-the part not covered by the cubes, it follows that $E_0\subset \mathcal{S}$.

We define $E=\pi_k^{-1}(E_0)$, and $\Gamma_j= \pi_k^{-1}( \mathcal{C}_j) $, for $j=1,2,...,M$. Using that $\pi_k$ is a bijection, and the mentioned properties of $E_0$, and the family of cubes $\mathcal{C}$, the assertions $(i)-(iii)$ follow immediately from inequality (\ref{proj-distortion-ineq}). The proof is now complete.
\end{proof}

%$\newline$
\subsection{The Poisson kernel}

For $x\in D$ and $y\in \Gamma$ we denote by $P(x,y)$ the Poisson kernel corresponding to operator $\mathcal{L}$ in  $D$. It is proved in Lemma 2 of \cite{AL}, in a more general setting, that for all $x\in D$,
\begin{equation}\label{Lem-Poisson-est-with dist}
|P(x,y)|\leq C \frac{d(x)}{|x-y|^d }, \ y \text{ a.e. in } \Gamma,
\end{equation}
where the $y$ null set is independent of $x$. We remark here that the estimate (\ref{Lem-Poisson-est-with dist}) is proved in the case when the matrix $A$ is periodic. It is easy to see that our case can be reduced to the setting of \cite{AL} since we are only interested in values of $A$ on a bounded region $D$.

%The next Lemma will be used to control the contribution of small pieces of the boundary in the value of solutions to (\ref{problem-formulation}).
\begin{lem}\label{Lem-Poisson-on a small strip}
Let  $\rho>0$ be a small number, $x\in D$ be fixed with $d(x)\geq 2\rho$, and let $\Pi$ be one of the faces of $D$.  Then, there exists a constant $C$, independent of $x$  and $\rho$ such that
$$
\int\limits_{\Pi_{\rho}} |P(x,y)| d\sigma(y) \leq C \frac{\rho}{d(x)}.
$$
\end{lem}

\begin{proof}
If $d=2$ then $\Pi$ is a segment, and $\Pi_\rho$ is a union of two segments of size $\rho$. It follows from (\ref{Lem-Poisson-est-with dist}) that
$$
\int\limits_{\Pi_{\rho}} |P(x,y)| d\sigma(y) \leq C  \frac{1 }{ d(x) } \int\limits_{\Pi_\rho } d\sigma(y) \leq C \frac{\rho}{d(x)}.
$$

We now consider the case  $d\geq 3$. Note that the boundary of $\Pi$ is a subset of $(d-2)$-dimensional boundary of $D$, that is its edges.

We will use the formula (\ref{Lem-Poisson-est-with dist}) to estimate the Poisson kernel, and for this reason we start with the analysis of the level sets $A(r):=\{ y\in \Pi_\rho:  \ |x-y|=r \}$, for $r \geq d(x)$. Since $D$ is bounded, without loss of generality we will assume that $\mathrm{diam}(D)\leq 1$. Observe that $A(r)$ is a $(d-2)$-dimensional object, which lies in the intersection of a boundary of a $d$-dimensional ball centered at $x$ and having radius $r$ with a plane, and then a small strip of size $\rho $ in that plane. It follows that its $(d-2)$-dimensional measure will be bounded by a surface measure of intersection of a sphere in $\R^{d-1}$ with a strip of size $\rho$ in $\R^{d-1}$, taken the maximum of all such intersections. Using the symmetries of a sphere it is easy to see that the maximum is attained when the center of a sphere is on the same distance from the bounding hyperplanes of the strip. Taking this into account consider in $\R^{d-1}$ the following subset $E_r :=B(0,r) \cap \{|x_1|\leq \rho \} $, where $x=(x_1,...,x_{d-1})$. We need to estimate $\mathcal{H}^{d-2}(\partial B(0,r) \cap \{|x_1|\leq \rho \} )$, which is clearly less or equal to $\frac{d}{dr} \mathcal{H}^{d-1}(E_r)$. Integrating over $(d-2)$-dimensional spheres (slices parallel to the cutting hyperplanes) we get
$$
\mathcal{H}^{d-1}(E_r)=C\int\limits_{0}^{\rho} (r^2-x^2)^{\frac{d-2}{2}} dx:=\mathcal{I}_d(r).
$$
Differentiating the last expression with respect to $r$, and using the fact that $r\geq d(x) \geq 2 \rho$ we obtain
$$
\frac{d}{dr} \mathcal{I}_d(r)= C r \int\limits_0^{\rho} (r^2 - x^2)^{\frac{d-4}{2}} dx \leq C r^{d-3} \rho.
$$
Hence we conclude that
\begin{equation}\label{est of A-r}
\mathcal{H}^{d-2}(A(r)) \leq C r^{d-3} \rho, \ r\geq d(x).
\end{equation}
Let $x_\Pi$ be the orthogonal projection of the point $x$ onto the plane containing $\Pi$, clearly we have $| x-x_\Pi |\geq d(x)$. After a rotation of the coordinates we may assume that $\Pi$ is contained in the plane $\{x_d=0 \}$. We write $\overline{y}=(y_1,...,y_{d-1},0)$ for the points in $\Pi$, and for $r>0$ denote by $\mathbb{S}(x_\Pi, r)$ the boundary of $d-1$-dimensional ball in $\Pi$ with center $x_\Pi$ and radius $r$. Now using (\ref{Lem-Poisson-est-with dist}) and integrating the Poisson kernel in the spherical coordinates, we obtain
\begin{equation}\label{Poisson-in spherical}
\int\limits_{\Pi_{\rho}} |P(x,y)| d\sigma(y) \leq C d(x) \int\limits_0^1  \int\limits_{ \mathbb{S}(x_\Pi, r)\cap \Pi_{\rho}  } \frac{ d \mathcal{H}^{d-2}(\overline{y}) }{| x-\overline{y}  |^d}       dr,
\end{equation}
Next, if $\overline{y} \in \Pi_\rho$ with $| \overline{y} - x_\Pi | =r $, then we have $|x - \overline{y} |^2 = |x-x_\Pi  |^2 + | x_\Pi  - \overline{y} |^2 \geq d(x)^2 + r^2    $, and also $ \overline{y} \in A( [ r^2 + |x-x_\Pi|^2 ]^{1/2} ) $. Using these, from (\ref{Poisson-in spherical}) and (\ref{est of A-r}) we obtain
\begin{equation}\label{Poisson-on small strip}
\int\limits_{\Pi_{\rho}} |P(x,y)| d\sigma(y) \leq C d(x) \int\limits_0^1  \frac{  \rho  (  d(x)^2+ r^2 )^{\frac{d-3}{2}}       }{ ( d(x)^2 + r^2  )^{\frac d2 }  }     dr   = C  d(x) \rho \int\limits_0^1  \frac{dr}{   (  d(x)^2 + r^2  )^{\frac 32} }  .
\end{equation}
To estimate the last integral, we set $a=d(x)$, then
\begin{multline*}
\int\limits_0^1  \frac{dr}{   (  d(x)^2 + r^2  )^{3/2} } =\int\limits_0^1 \frac{d r} {  r^3 \left( 1+\frac{a^2}{r^2}  \right)^{3/2}  }  = ( \text{ setting } y=r^{-2} ) \\
\frac 12 \int\limits_1^{\infty} \frac{dy}{  (1+a^2 y )^{3/2} }  = (  \text{ setting }  z=\sqrt{1+a^2 y} ) =  \frac{1}{a^2} \int\limits_{ \sqrt{1+a^2} }^{\infty} \frac{ dz } {z^2} = \\
\frac{1}{ \sqrt{1+a^2} } \frac{1}{a^2} \leq C \frac{1}{d(x)^2}.
\end{multline*}
This, together with (\ref{Poisson-on small strip}) completes the proof.
\end{proof}

In the next Lemma we prove certain type of H\"{o}lder-smoothness for $P(x,y)$ with respect to its boundary variable $y$ and uniformly in $x$.
We shall also define $\Gamma^*$ to be the set of singular boundary points (see Appendix).

\begin{lem}\label{Lem-Poisson is Holder}
Retain the hypothesis of the Standing Assumptions in Section \ref{ass}, and if $\alpha_*>1$ set $\beta=1$, otherwise, let $0<\beta <\alpha_*$ be any number. Fix any $\delta\geq 0$, $x\in D$, and $y_1, y_2\in \Pi \setminus \Gamma^*$, where $\Pi$ is a face of $D$, and $|y_1 - y_2| \leq c d(x)$, where $c$ is some universal constant. Then, there exists a constant $C$ depending on $\beta$, and $\delta$, and independent of $x,y_1,y_2$ such that
$$
\left|  P (x,y_1) - P(x,y_2)  \right| \leq C \frac{|y_1-y_2|^{\beta}}{|x-y_1|^{d-1+\beta+\delta}},
$$
where $\delta$ can be taken arbitrarily small positive non zero number in dimension two, and zero in dimensions greater than two.
\end{lem}

\begin{proof} Let $G(x,y)$ be the Green's function corresponding to problem (\ref{problem-formulation}), then the Poisson kernel has the representation $P(x,y)=\nu_y A(y)\nabla_y G(x,y)$, where $\nu_y$ is the outward unit normal of $\Gamma$ at $y$. We will study the regularity properties of the Green's function, which together with smoothness of $A$ will imply the result. We will need the following estimates on the Green's function of $\mathcal{L}$,
\begin{equation}\label{in Lem-Poisson-Hold-Green point-est}
|G(x,y)|\leq C   \begin{cases}  \log \frac{1}{|x-y|}, &\text{  $d=2$}, \\
  | x-y |^{2-d} &\text{  $d \geq  3$ },   \end{cases}
\end{equation}
for all $(x,y)\in D\times D$, with $x\neq y$, where for $d=2$ the estimate is proved in \cite{DoMu}, and for $d\geq 3$ in \cite{GW}. Now fix any two points $x_0,y_0\in D$, and set $R=|x_0-y_0|$, $D_R=\frac 1R (D- x_0)$, and let $G_R(\cdot, \cdot)$ be the Green's function for the scaled domain and the scaled operator. Clearly $G_R(w,z)=R^{d-2} G( Rw+x_0, Rz+x_0  )$, where $w, z\in D_R$. Consider $h_R(z):=G_R(0,z) $ in the set $\tilde D_R:=D_R \cap (B_4(0)\setminus B_{1/4}(0))$. Then $h_R$ is a solution to our PDE in this set and zero on $\partial \tilde D_R \setminus \overline{(B_4(0)\setminus B_{1/4}(0))}$. We claim that
\begin{equation}\label{C1b1}
h_R\in  C^{1,\beta} (D_R \cap ( B_{3}(0)\setminus B_{1/2}(0)) )
\end{equation}
with uniform norm bounded by constant times the supremum norm of $h_R$ on the set $\tilde D_R$. In the sequel, when proving (\ref{C1b1}) we will keep in mind the mentioned relation of constants with the supremum norm of $h_R$.

We first show that (\ref{C1b1}) with (\ref{in Lem-Poisson-Hold-Green point-est}) would imply the desired estimate. Take any $y_1, y_2 \in \Pi \setminus \Gamma^*$ with $| y_1 - y_2|\leq C d(x_0)  $. Since $n(y_1)=n(y_2)$, from the Poisson representation we have
\begin{multline}\label{Poisson-diff}
| P(x,y_1)  - P(x,y_2 ) | \leq   |  n(y_1) ( A(y_1) - A(y_2) ) \nabla_y G(x_0, y_1)|  + \\ |n(y_1) A(y_2) ( \nabla_y   G(x_0,y_1) - \nabla_y G(x_0,y_2)   )      |.
\end{multline}
On the other hand for $R=|x_0-y_0|$, and $z\in D_R \cap ( B_{3}(0)\setminus B_{1/2}(0)) $ we have
\begin{equation}\label{grad h-R}
\nabla h_R(z)= R^{d-1} \nabla_y G(x_0, y) , \text { where } y=Rz+x_0.
\end{equation}
It is then easy to see that (\ref{Poisson-diff}), (\ref{grad h-R}), (\ref{C1b1}) and (\ref{in Lem-Poisson-Hold-Green point-est}), together with the smoothness of $A$ would imply the desired estimate. We just remark that in dimension two we may tradeoff the logarithmic singularity in the supremum norm of $h_R$ by slightly increasing the power in the denominator of the estimate in the Lemma by means of the small parameter $\delta$ introduced in the formulation, while in dimensions greater than two, the supremum norm of $h_R$ is uniformly bounded away from the origin.

In what follows we prove (\ref{C1b1}). Observe that due in any compact and do not specify where to Schauder estimates
\begin{equation}\label{C1b2}
h_R\in C^{1,\beta}  ( D_R \cap (B_{3}(0) \setminus B_{1/2}(0)) )
\end{equation}
It remains to show that when approaching the boundary of $\tilde D_R$ the norm does not blow-up.

From boundary regularity for elliptic equations, we also know that solutions are smooth at regular boundaries (see Theorem 6.19 in \cite{GT}).
In particular in our case we have (at least) $C^{2}$ regularity for $h_R$ on the flat boundaries, $\partial D_R \setminus \partial^* D_R$, where $\partial^* D_R$ denotes the set of all points of the boundary of $D_R$ that belong to more than one face of $D$, i.e. the corner points. Again the norm may blow up when approaching the corners $\partial^* D_R$. Since we can approach the corner points both tangentially and non-tangentially, we
may consider two cases for  $x_j \to \partial^* D_R$:
\begin{center}
 (i) non-tangential to the boundary, \qquad   (ii) tangential to the boundary.
\end{center}
For (i)  we consider two points $y_i$ ($i=1,2$), with the property that they approach $\partial^* D_R$ non-tangentially. Then
if $|y_1-y_2| \geq (1/4) dist (y_1, \partial^* D_R)$ then by Lemma \ref{gradient-estimate}
$$
|\nabla h_R (y_1) - \nabla h_R (y_2) | \leq |\nabla h_R (y_1)| + | \nabla h_R (y_2) | \leq
$$
$$
 C \max_{i=1,2} dist^\beta (y_i, \partial^* D_R)
 \leq C |y_1-y_2|^\beta \ .
 $$
If $|y_1-y_2| \leq (1/3) dist (y_1, \partial^* D_R)$ then we scale $h_R$ at $y_1$ with the distance to the corner
$\tilde h_R (y)= h_R (y_1 + d_1 y)/ d_1^{1+\beta}$, where $d_1$ is the distance from $y_1$ to $\partial^* D_R$. By Lemma \ref{barrier}
we have $\tilde h_R$ is uniformly bounded in $B_1$ and that $\tilde y_2 = (y_2-y_1)/d_1 \in B_{1/3}(0)$. Since in $B_{1/2}(0)$ we have no
corner points but only smooth boundary, the elliptic regularity implies that $\tilde h_R$ is uniformly $C^2$, say, (independent of $y_1, y_2$).
But then  the $C^{1,\beta}$ norm of $\tilde h_R$ is uniformly bounded (independent of $y_1, y_2$), and we have the same for $h_R$. In particular
 \begin{multline*}
 |\nabla h_R (y_1) - \nabla h_R (y_2) | =  d_1^{\beta} |\nabla  \tilde h_R (0)  -  \nabla \tilde  h_R (\tilde y_2) | \leq
 C   d_1^{\beta} |\tilde y_2|^\beta = C  |y_2-y_1|^\beta \ .
 \end{multline*}

For (ii) we start by taking any point $z_0$
on the flat boundary and consider the half ball $B^+_s(z_0)$ which is inside the domain $\tilde D_R$.
For simplicity assume that the flat portion
of the boundary, with $z_0$ on it, is part of the hyperplane $\{x_d =0\}$, such that $B_s^+=\{x_d >0\} \cap B_s(z_0)$. Now we let $s$ denote  the largest real number such that $B^+_{2s}(z_0) \subset \tilde D_R$. Obviously    $\partial^* D_R \cap \overline{B^+_s(z_0)}=\emptyset  $, and
\begin{equation}\label{ss}
c_0 s \geq   dist(z_0, \partial^* D_R)
\end{equation}
for some $c_0>0$, due to Lipschitz character of the domain. Invoking Lemma \ref{barrier} and using (\ref{ss}) we have that for $ z \in B^+_1(0)$ the function $v_s(z):=h_R(s z + z_0)/s^{1+\beta}$
 satisfies the bound
\begin{multline*}
 0\leq v_s (z) \leq C \frac{\left(\hbox{dist}(s z + z_0, \partial^* D_R ) \right)^{1+\beta}}{s^{1+\beta}} \leq
C\frac{\left(\hbox{dist}( z_0, \partial^* D_R ) + s \right)^{1+\beta}}{s^{1+\beta}} \leq \\
  C(c_0 + 1)^{1+\beta}
\end{multline*}
which is uniformly bounded in $B^+_1(0)$.
Hence classical Schauder estimates can be applied to  conclude uniform  $C^{1,\beta}$-estimates for $v_s$ in $B^+_{1/2}(0)$, i.e.
 $$| h_R |_{C^{1,\beta}} (B^+_{s/2}(z_0)) = | v_s |_{C^{1,\beta}} (B^+_{1/2}(z_0)) \leq C_0 .$$

This in particular means that the $C^{1,\beta}$ norm is uniformly bounded up to any flat boundary points, which is the desired result.

\end{proof}

\section{Proofs of the theorems}

\noindent \textbf{Proof of Theorem \ref{Thm-Pointwise}}
By the Poisson representation we have
\begin{multline*}
u_{\e}(x)-u_0(x)=\int\limits_{\Gamma}P(x,y)[g_{\e}(y)-\overline{g}(y)] d\sigma(y)= \\
\sum\limits_{j=1}^N \int\limits_{\Pi_j}P(x,y)[g_{\e}(y)-\overline{g}(y)] d\sigma(y),
\end{multline*}
hence it is enough to study the integrals over one particular face. Let $\Pi$ be one of the faces of $\Pi$ with Diophantine normal vector $\nu \in \Omega(\tau, c)$. We will assume that the boundary data $g$ is smooth of order greater than $\frac{d-1}{2}+(d-1)\tau$. Since $g$ is smooth and 1-periodic we have
$$
g(y)=\sum\limits_{m\in \Z^d} c_m e^{2 \pi i m \cdot y},
$$
and the order of smoothness of $g$ assures that the  series converges absolutely. Define $\mathcal{I}_1=\{m\in \Z^d: m_1\neq 0 \}$ and for $k=2,3,...,d$ set $\mathcal{I}_k=\{m\in \Z^d: m_k\neq 0 \}\setminus ( \mathcal{I}_1 \cup ... \cup \mathcal{I}_{k-1})$.
We get
$$
\int\limits_{\Pi} P(x,y)[g_{\e}(y)-\overline{g}(y)] d\sigma(y)= \sum\limits_{k=1}^d \sum\limits_{m\in \mathcal{I}_k} c_m \int\limits_{\Pi} P(x,y) e^{ \frac{2\pi i}{\e} m\cdot y  } d\sigma(y).
$$

We fix $x\in D$, $1\leq k \leq d$, and a small parameter $0<\rho \leq c d(x)$, where the constant $c$ will be chosen from (\ref{est on E}) below. Applying Lemma \ref{Lem-Partition by lattice} we get a set $E\subset \Pi$, and a family $\{ \Gamma_j^{\rho} \}_{j=1}^M$ with properties $(i)-(iii)$ of the Lemma, and let $c_0$ be the constant from part $(i)$. Since $E\subset \Pi_{c_0 \rho}$ from Lemma \ref{Lem-Poisson-on a small strip} we get
\begin{equation}\label{est on E}
\int\limits_{E} |P(x,y) | d\sigma(y) \leq C \frac{\rho}{d(x)}, \text{ for } x\in D \text{ with } d(x) \geq 2c_0 \rho.
\end{equation}
Now for $j=1,2,...,M$ fix some $y_j\in \Gamma_j^\rho$, and outside $E$ we have
\begin{multline*}
\int\limits_{\Pi \setminus E} P(x,y) e^{ \frac{2\pi i }{\e} m \cdot y } d\sigma(y) = \sum\limits_{j=1}^M \int\limits_{\Gamma_j^\rho} [P(x,y) - P(x,y_j)] e^{ \frac{2\pi i }{\e} m \cdot y } d\sigma(y)+ \\
\sum\limits_{j=1}^M P(x,y_j) \int\limits_{\Gamma_j^\rho}  e^{ \frac{2\pi i }{\e} m \cdot y } d\sigma(y):=A_1(x)+A_2(x).
\end{multline*}

\noindent \textbf{Estimate of $A_1$.} Since $\mathrm{diam}(\Gamma_j^{\rho})\leq C d(x)$, for any $y\in \Gamma_j^{\rho}$ from Lemma \ref{Lem-Poisson is Holder} we obtain
$$
|P(x,y)-P(x,y_j)| \leq C \frac{|y-y_j|^{\beta}}{|x-y_j|^{d-1+\beta + \delta/2 }}.
$$
In view of $|y-y_j|\leq \mathrm{diam}(\Gamma_j^{\rho}) \leq C \rho $, the last estimate implies
\begin{equation}\label{est-A1-first}
|A_1(x)| \leq C \sum\limits_{j} \int\limits_{\Gamma_j^{\rho}} \frac{|y-y_j|^{\beta} }{|x-y_j|^{d-1+\beta +\delta/2 }} d\sigma(y) \leq C \frac{\rho^{\beta}}{d(x)^{\beta+\delta}} \sum\limits_j \frac{|\Gamma_j^{\rho}|}{|x-y_j|^{d-1-\delta/2 }},
\end{equation}
where $\delta>0$ is any small number. The sum in $(\ref{est-A1-first})$ is bounded up to multiplication by some constant depending on $\delta>0$ by the integral $\int\limits_{\Gamma}\frac{d\sigma(y)}{|x-y|^{d-1-\delta/2}}$, and hence is uniformly bounded with respect to $x$. We conclude that
\begin{equation}\label{est-A1-final}
|A_1(x)|\leq C_{\delta} \frac{\rho^{\beta}}{d(x)^{\beta+\delta}}.
\end{equation}

\noindent \textbf{Estimate of $A_2$.} Observe that $m_k \neq 0$, and $\pi_k(\Gamma_j^\rho)$ is a $(d-1)$-dimensional rectangle with sides parallel to the coordinate axes, hence we may apply Lemma \ref{Lem-Diophantine-Decay}, and using the fact that $\mathcal{H}^{d-1}(\Gamma_j^\rho) \approx \rho^{d-1}$ we get
\begin{multline*}
\left|  \int\limits_{\Gamma_j^\rho} e^{ \frac{2\pi i}{\e} m \cdot y } d\sigma(y) \right| \leq C \e^{d-1} || m ||^{(d-1) \tau(\nu)} \leq \\  C \left(\frac{\e}{\rho}\right)^{d-1} \mathcal{H}^{d-1}(\Gamma_j^{\rho}) || m ||^{(d-1) \tau(\nu)}.
\end{multline*}
Using this for $A_2$ we have
$$
|A_2(x) |\leq C \left(\frac{\e}{\rho}\right)^{d-1} || m ||^{(d-1)\tau} \sum\limits_j |P(x,y_j)| \mathcal{H}^{d-1}(\Gamma_j^{\rho}) \leq C \left(\frac{\e}{\rho}\right)^{d-1} || m ||^{(d-1)\tau}.
$$

Combining the estimates for $A_1$ and $A_2$, for the integral on $\Pi\setminus E$ we get
\begin{multline}\label{est-outside E}
\left| \int\limits_{\Pi\setminus E} P(x,y) [ g_\e (y) - \overline{g}  ] d \sigma(y) \right| \leq \\ C \sum\limits_{k=1}^d \sum\limits_{m\in \mathcal{I}_k } |c_m| \left( \frac{ \rho^\beta}{d(x)^{\beta+\delta}} + \left( \frac{\e}{ \rho  } \right)^{d-1} || m ||^{(d-1) \tau }  \right) \leq C \left( \frac{\rho^{\beta} }{d(x)^{\beta+ \delta}} + \left( \frac{\e}{\rho} \right)^{d-1} \right),
\end{multline}
where the convergence of series with Fourier coefficients is due to the smoothness of $g$ of order greater than $\frac d2 + (d-1) \tau$ (see Lemma 2.3, \cite{ASS}). Since $\beta \leq 1$ clearly the estimate (\ref{est on E}) is better than (\ref{est-outside E}), thus we have
\begin{equation}\label{u-e-u-0}
|u_\e (x) - u_0(x) | \leq C_{\delta} \left(  \frac{\rho^{\beta}}{d(x)^{\beta+\delta}} + \left(\frac{\e}{\rho}\right)^{d-1}  \right),
\end{equation}
for all $x\in D$ satisfying $d(x)\geq 2 c_0 \rho$. Equalizing the estimates we obtain
$$
\frac{\rho^{\beta}}{d(x)^{\beta+\delta}} = \left(\frac{\e}{\rho}\right)^{d-1} \Longleftrightarrow \rho=\e^{\frac{d-1}{d-1+\beta}} d(x)^{\frac{\beta+\delta}{d-1+\beta}} .
$$
Comparing this with $d(x)\geq 2c_0 \rho$, we get that (\ref{u-e-u-0}) holds true if $d(x) \geq C \e^{ \frac{d-1}{d-1-\delta} }$, where $C$ is some absolute constant, thus we conclude that
$$
|u_\e (x) - u_0(x) | \leq C_{\delta} \left(  \frac{\e^{\beta}}{d(x)^{\beta+\delta}} \right)^{\frac{d-1}{d-1+\beta}}.
$$
When $d(x) < C \e^{ \frac{d-1}{d-1-\delta} }$ the estimate of the theorem follows by the uniform boundedness of $|u_\e-u_0|$.
Theorem \ref{Thm-Pointwise} is proved.

$\newline$

\noindent \textbf{Proof of Theorem \ref{Thm-Lp}}. For $\beta>0$ we set $\kappa= \frac{d-1}{d-1+\beta}  $. By Theorem \ref{Thm-Pointwise} we have
\begin{equation}\label{est-pointwise thm2}
|u_\e(x) - u_0(x) | \leq C \frac{\e^{\beta \kappa}}{d(x)^{ (\beta+\delta) \kappa }}, \ x\in D.
\end{equation}
Set $p_0=\frac{1}{\beta \kappa}$, and fix $1\leq p<p_0$. Then for $\delta>0$ small enough we have $p(\beta+ \delta) \kappa = p \beta \kappa + \delta p \kappa<1$. This, together with (\ref{est-pointwise thm2}) implies that
\begin{equation}\label{est-Lp up to p0}
|| u_\e - u_0 ||_{L^p(D)} \leq C \e^{\beta \kappa}, \ 1\leq p<p_0.
\end{equation}
Now fix $p_0 \leq r<\infty$, and let $1\leq p < p_0$. Using the uniform boundedness of $|u_\e - u_0| $, and estimate (\ref{est-Lp up to p0}) we obtain
\begin{multline*}
|| u_\e - u_0 ||_{L^r (D)} = \left( \int\limits_D | u_\e - u_0 |^{r-p} | u_\e - u_0|^p    \right)^{ \frac 1r } \leq  C || u_\e - u_0 ||_{L^p(D)}^{\frac pr } \leq \\ C \e^{\frac{\beta \kappa p}{r} }.
\end{multline*}
Now take $p=p_0-\delta$, where $\delta>0$ is small enough. Since $p_0 \beta \kappa=1$, from the last estimate we get
$$
|| u_\e - u_0 ||_{L^r (D)} \leq C \e^{ \beta \kappa  \frac{ p_0 - \delta }{r}  } = C \e^{ \frac{1-\beta \kappa \delta}{r} } = C \e^{ \frac 1r - \delta_1 },
$$
where $\delta_1= \frac{\beta \kappa \delta}{r}$. Combining this with (\ref{est-Lp up to p0}), for $1\leq p <\infty$ we get
\begin{equation}\label{est-with-min}
|| u_\e -u_0 ||_{L^p(D)} \leq C \e^{ \min\{\beta \kappa, \frac 1p \}  - \delta }.
\end{equation}
Now if $\beta=1$, then we are done, otherwise we have $\alpha_*\leq 1$, and (\ref{est-with-min}) holds true for each $0<\beta<\alpha_*$, and $\delta>0$. Observe that for all $d\geq 2$ we have
$$
0<\alpha_* \kappa -\beta \kappa <\alpha_*-\beta, \text{ where }   0<\beta<\alpha_* \leq 1.
$$
Using this, for each $\delta>0$ we choose $0<\beta<\alpha_*$ such that $\alpha_*-\beta<\delta/2$, and from (\ref{est-with-min}) we get
$$
|| u_\e -u_0 ||_{L^p(D)} \leq C \e^{ \min\{\gamma, \frac 1p \}  - \frac 32 \delta  },
$$
completing the proof.

$\newline$
\noindent \textbf{Proof of Theorem \ref{Thm-Optimality} }. For the proof we will follow the same strategy as in Section 3 of \cite{ASS2}. The only part that needs to be modified in this setting is Lemma 3.2 of \cite{ASS2}, which proves certain type of equidistribution result for the family $\lambda \Gamma \mod 1$, as $\lambda \To \infty$, where for $x\in \R^d$, $x \mod 1 $ denotes the unique point $y\in \mathbb{T}^d$, with $x-y\in \Z^d$. On the other hand, the proof of Lemma 3.2 of \cite{ASS2} is based on the following fact: for any smooth function $g: \mathbb{T}^d \To \mathbb{C}$ one has
\begin{equation}\label{optimality-scales}
\int\limits_{\mathbb{T}^d} g(x) dx = \lim\limits_{\lambda \To \infty} \frac{1}{\mathcal{H}^{d-1} (\Gamma)  } \int\limits_{\Gamma} g(\lambda y ) d\sigma(y).
\end{equation}
So, to complete the proof of the Theorem we need to prove (\ref{optimality-scales}), which is now due to the Diophantine property of the faces of $D$. Observe that since the linear combinations of exponentials $e_m(y):=e^{2\pi i m\cdot y}$, $m\in \Z^d$, $y\in \mathbb{T}^d$ are dense in the uniform metric in the space of smooth functions on $\mathbb{T}^d$, it is enough to prove (\ref{optimality-scales}) for each $e_m$, $m\in \Z^d$. When $m=0$ then (\ref{optimality-scales}) is trivial, now fix some non zero $m\in \Z^d$. We need to show that the limit in (\ref{optimality-scales}) is 0, which is enough to establish on each face of $D$. Let $\Pi$ be a face of $D$ with a normal vector $\nu \in \Omega(\tau,c)$. The proof will be complete once we show that
\begin{equation}\label{optimality-scales-final}
\mathcal{J}_\lambda:=\int\limits_{\Pi} e_m(\lambda y) d\sigma(y) \To 0, \text{ as } \lambda \To \infty.
\end{equation}
Since $m\neq 0$, then $m_k\neq 0$ for some $1\leq k \leq d$. Take $\e>0$ small and apply Lemma \ref{Lem-Partition by lattice} for $k$ and $\e$. We will get a partition of $\Pi $ into a set $E$, and a finite family of sets $\{ \Gamma_j \}_{j=1}^M $ with properties $(i)-(iii)$ of Lemma \ref{Lem-Partition by lattice}. It is easy to see from the definition of sets $\Pi_{\e}$ that $\mathcal{H}^{d-1}(\Pi_\e) \leq C \e$, and since $E\subset \Pi_{c_0 \e}$, for some absolute constant $c_0$, we have $\mathcal{H}^{d-1}(E) \leq C \e$. We then use the properties of the partition and applying Lemma \ref{Lem-Diophantine-Decay} on each of the $\Gamma_j$'s we get
$$
|\mathcal{J}_\lambda|=\left| \int\limits_{E} + \sum\limits_{j=1}^M \int\limits_{\Gamma_j} \right| \leq  C \e + C \lambda^{-(d-1)} || m||^{(d-1)\tau} M \leq C \e,
$$
if $\lambda$ is large enough, which proves (\ref{optimality-scales-final}), completing the proof of the Theorem.

$\newline$

\section{Appendix: PDE tools}
In this appendix we shall prove some basic estimates for Green's function for a given second order elliptic linear operator $L$, in polygonal domains. The estimates are standard but hard to find in literatures, therefore for the readers' convenience we have chosen to give proofs of these estimates.

Our starting point will be to fix the domain $D$ and the operator $L$, as defined in Section \ref{sec-intro}, along with the corresponding Green's function $G(x, y)$ defined on $D\times D \setminus \{ (x,x): x\in D \}$.

By $\Gamma^*$ we denote the "singular" boundary of $D$, i.e. the set of all points of $\Gamma$ that belong to more than one face of $D$. Since each $j$-dimensional face, $0\leq j \leq d-2$, is sitting on the end of a $(j-1)$-dimensional face then each lower-dimensional face is also sitting on the intersection of two $(d-1)$-dimensional faces. In particular the behavior of the Green's function should be studied at intersection between any two faces, as these are the "worst" points, since due to convexity of $D$ the solution behaves better closer to lower dimensional faces. This in particular suggests that it is for us inevitable to avoid the lowest regularity properties of the solution in the vicinity of the "largest" possible angle, among the angle between two intersecting faces.

To be more precise let us fix a boundary point $z \in \Gamma^* $, and let $\Pi_1$ and $\Pi_2$ be any two supporting hyperplanes of $D$ at $z$. Choose $\alpha>0$ so that the angle between these two planes, i.e. $\arccos (\nu_1\cdot \nu_2)$ equals $ \pi /(1+\alpha)$, where $\nu_i$ denotes the outward unit normal to $\Pi_i$. Then obviously a rotated and translated version of the function $ \mathbf{Im} (x_1 + \mathbf{i} x_2)^{1+ \alpha}$ will be harmonic in the convex cone generated by the two planes. It is well known that positive harmonic functions in cone like domains (with zero boundary values) behave as $r^\lambda$ where $\lambda$ is the first eigenvalue to the Laplace-Beltrami operator of surface which is the intersection of the cone with the unit sphere (see e.g. \cite{Anc}). This fact can be used along with freezing coefficient techniques to show similar behavior for the solutions to variable coefficients elliptic equations.

We formalize the discussion above in the next lemma. Let $D$ be a given convex polygonal domain, and fix $x_0\in \Gamma^*$. Choose $\alpha>0$ so that $\pi / (1+\alpha)$ be the maximal angle between any two supporting planes of $D$ at the point $x_0$.

\begin{lem}\label{barrier}
With the above notation, consider any (nonnegative) solution $h$ to $\mathcal{L}h=0$ in $D\cap B_1(x_0)$ with zero boundary data on $B_1(x_0) \cap \partial D$, and non-negative on $D \cap \partial B_1(x_0)$. Then for any $\beta<\alpha$ there exists a constant $C$ depending on $\beta$ such that
$$
0\leq   h(x) \leq C  M  |x-x_0|^{1+\beta},  \qquad \forall \ x \in  D \cap B_1,
$$
where $M= \sup\limits_{ B_1(x_0)\cap D } h  $, and $x_0\in \Gamma^*$.
\end{lem}

\begin{remark} This estimate is well-known, but not easy to find a reference to (at least we could not!). Indeed, the estimate should be much sharper than what we present here, but that will not affect our results, as the estimate deteriorates at faces of $(d-2)$-dimension (facets), and the only optimality we loose (by our statement) is that we do not allow $\beta = \alpha$. The latter is due to our proof.
Variations of this lemma can be found in \cite{K}, and \cite{Anc}.
\end{remark}

\begin{proof}
The proof is based on scaling and Phragm\'en-Lindel\"{o}f type argument. After a translation we may assume $x_0= 0$. Next, if $A$ is the matrix of the operator $\mathcal L$, then after a change of variables by $x=By$, where $B$ is an invertible matrix of size $d$, the matrix, corresponding to the new operator will be $| det B |^{-1} B^T A B $. Also, note that the matrix $\frac 12 (A(0)+A(0)^T ) $ is positive definite and symmetric, hence by a
 composition of orthogonal transformation and scaling we may bring it to a scalar multiple of an identity matrix, i.e. the symmetric component of the new operator will be a scalar multiple of Laplacian at the origin. Since the orthogonal transformation and scaling will transform $D$ to a new polygonal domain, with the same angles between its faces, as the original one, without loss of generality, we will assume that $\frac 12 (A(0)+A(0)^T )$ is the identity matrix.

Let $\Pi_i=\{x\in \R^d: x\cdot \nu_i =0 \}$, $i=1,2$ be two supporting planes to $D$ at the origin, so that the angle between $\Pi_1$ and $\Pi_2$ is $\alpha$. Set $D_\alpha=\{x\in \R^d: x\cdot \nu_i>0, \ i=1,2 \}$, then clearly $D\subset D_\alpha$. Now, for any $\gamma \in (\beta, \alpha)$ we denote by $D_\gamma$ a convex region containing $D_\alpha$, bounded by two hyperplanes passing through the origin and forming an angle equal to $\pi/(1+\gamma)$. Let us finally set $H_\gamma$ to be the positive barrier function supported in $D_\gamma$, which is a rotation of  $ \mathbf{Im} (x_1 + \mathbf{i} x_2)^{1+ \gamma}$. Clearly for some constant $C$ we have
\begin{equation}\label{H}
   \sup_{B_R \cap D_\gamma } H_\gamma (x) = CR^{1+\gamma}.
 \end{equation}
Also, to simplify notation we define the solutions $h$ to be zero outside $D$. After this preliminary set up, we claim now that there exists a constant $C_0 >0$ such that
\begin{equation}\label{claim}
 \sup_{B_r} h(x) \leq C_0 M r^{1+\beta},  \  \forall  r\in (0,1],  \hbox{ where } M=\sup_{B_1}h.
 \end{equation}
% For simplicity of the exposition we extend both $h$, and $H$ as zero functions outside their domains $D$.
If this fails, then there exists a sequence of points $r_j \searrow  0$, positive numbers $c_j \To \infty$, and solutions $h_j$ to our equation such that
\begin{equation}\label{scaling}
 \sup_{B_{r_j}  }h_j=c_j M_j r_j^{1+\beta},
\end{equation}
and
\begin{equation}\label{growth}
  \sup_{B_{r}  }h_j < c_j M_j r^{1+\beta}, \  \forall  r\in (r_j  ,1],
\end{equation}
where $M_j=\sup_{B_1}h_j$. To show this, we proceed by induction. Indeed, if (\ref{claim}) is false, then for $c_1=2$ there exists a solution $h_1 $ with $\sup_{B_r} h_1 \geq c_1 M_1 r^{1+\beta} $, for some $0<r<1$. We now take $r_1 $ to be the largest of these $r$, hence we get
$$
\sup\limits_{B_r} h_1 \leq c_1 M_1 r^{1+\beta}, \  \forall r\in ( r_1, 1],
$$
and
$$
\sup\limits_{B_{r_1}} h_1 = c_1 M_1 r_1^{1+\beta}.
$$
Now if we have chosen $r_j$, $c_j$, and $h_j$ satisfying (\ref{scaling}) and (\ref{growth}), for $j=1,2,..., n$, we take $c_{n+1}>c_n+1$ so that $c_{ n+1 } \left( \frac 12 r_n \right)^{1+\beta}  >1$. Then we proceed as in the case $n=1$. Clearly we will get $r_j$ decreasing to 0.

Scaling $h_j$ by $r_j$ through $\tilde h_j (x)= h_j(r_j x) / c_j M_j r_j^{1+\beta}$, we see from (\ref{scaling}) and (\ref{growth}) that
\begin{equation}\label{growth1}
1\leq \sup_{ B_R  } \tilde h_j \leq R^{1+\beta} \qquad \forall  1 \leq R \leq \frac{1}{r_j}.
\end{equation}
Furthermore, $\tilde h_j$ satisfies the scaled equation $\mathcal{L}_j \tilde h_j =0$ in the scaled domain $\frac{1}{r_j} (B_1 \cap D)$, and with zero boundary data on $ \frac{1}{r_j} ( \partial D \cap B_1 )$.

By compactness (or Arzel\'{a}-Ascoli type theorem) we can take a locally converging subsequence (again labeled $r_j$) such that
$$
\tilde h_j \to \tilde h_0,  \text{ and }  \mathcal{L}_j \to \mathcal{L}_0,
$$
where $\mathcal{L}_0$ is the operator with the constant matrix $A(0)$, and $\mathcal{L}_0 \tilde h_0=0$, in the cone $D_0:= \bigcup\limits_{j=1}^{\infty}  \frac{1}{r_j} ( D \cap B_1) $. Since $\frac{1}{2} ( A(0)+A(0)^T )$ is the identity matrix, we get that $\tilde h_0$ is harmonic in $D_0$. Moreover by (\ref{growth1}) we also have
\begin{equation}\label{growth2}
1\leq \sup_{B_R \cap D_0  } \tilde h_0 \leq R^{1+\beta}, \quad  \forall    R \geq 1.
\end{equation}
Now the blow-up cone $D_0$ (with vertex at the origin) whose boundary  consists of $k$-hyperplanes, (for some positive integer $k$) may be cylindrical (i.e. translation invariant) in some directions.
In this case we want to reduce the dimension by showing that the function $\tilde h_0$ is independent of the cylindrical direction. It should be remarked that such a reduction is needed only because of our barrier argument to follow; the argument does not work with cylindrical domains, and needs the cone to have only one vertex. One may see this as asking for the the intersection of the cone and the unit sphere to be a proper subset of the upper hemisphere (after rotation).

To this end we claim that positive harmonic functions in cones (with vertex at the origin) with zero Dirichlet data on the boundary of the cone must be homogeneous of some fixed positive degree. This is proved in Theorem 1 of \cite{K} for $NTA$-domains (non-tangentially accessible), and since Lipschitz domains are $NTA$, we get the claim for $D_0$ (for $NTA$-domains see \cite{K}, and the references therein). Next, we show that the solution $\tilde h_0$ is independent of the cylindrical directions. For simplicity, assume that $D_0$ is cylindrical  with respect to the last coordinate. Set $e_d = (0,..., 0 ,1) \in \R^d$, then for any $a>0$ we have that $\tilde h_1(x):= \tilde h_0(x+ae_d)$ is also a positive harmonic function in $D_0$ with zero Dirichlet data on the boundary, and hence is homogeneous of the same degree as $\tilde h_0$, say $p>0$. Now for any $\lambda>0$ we get
$$
\lambda^p \tilde h_1 (x) = \tilde h_1(\lambda x) = \tilde h_0 (\lambda x + a e_d ) = \lambda^p \tilde h_0 (x+ \frac{a}{\lambda} e_d) ,
$$
hence $ \tilde h_0 (x+ a e_d) = \tilde h_0 (x  + \frac{a}{\lambda} e_d ) \To \tilde h_0(x) $, as $\lambda \To \infty$. Thus $\tilde h_0$ is independent of the cylindrical directions. In particular, and without loss of generality, we may assume that our
 cone $D_0$ has the origin as the only vertex. This means that
\begin{equation}\label{inclusion}
 \partial B_1(0) \cap  \overline  D_0 \subset \partial B_1(0) \cap  \overline  D_\alpha \subset  \partial B_1(0) \cap    D_\alpha .
\end{equation}
Let us now take the two-dimensional barrier $H_\gamma$ in the convex (cylindrical) cone $D_\gamma$ introduced in (\ref{H}). Now choose $\e >0$ such that $\beta + \e < \gamma$. Define a new function $ H^\e_\gamma:=R^{-\epsilon} H_\gamma$, and observe  that there is a $c_0>0$ such that   $H_\gamma (x) \geq c_0 $
over the set $\partial B_1(0) \cap  \overline  D_0 $ (by Harnack's inequality). From this we infer that for $R$ sufficiently large
$$
 \inf_{D_0 \cap \partial B_R} H^\e_\gamma (x) = R^{1+\gamma -\e} \inf_{D_0 \cap \partial B_1} H^\e_\gamma (x)\geq c_0R^{1+\gamma -\e} >  R^{1+\beta} \geq \sup_{D_0 \cap \partial B_R }\tilde  h_0.
$$
Hence by the maximum principle (both functions are harmonic) we conclude that $H^\e_\gamma \geq \tilde h_0$ in the truncated cone $D_0 \cap B_R$. In particular as $R  $ becomes large  we arrive at $1= \sup_{B_1}\tilde h_0 \leq \sup_{B_1}H^\e_\gamma \leq R ^{-\e} \sup_{B_1}H_\gamma < 1/2$ (say). This is a contradiction and we conclude that our claim (\ref{claim}) must be true.
The proof of the lemma is complete.
\end{proof}

Using this lemma we can now estimate the gradient of the Green's function.

\begin{lem}\label{gradient-estimate}
Let $D$, and $h$ be as in Lemma \ref{barrier}. Then, for any $\beta<\alpha$ there exists a constant $C$ depending on $\beta$, so that
$$
|\nabla h(x)| \leq C_0 M d_{*} (x)^\beta,  \ \forall x\in D\cap B_{1/2},
$$
where $M=\sup\limits_{D \cap B_1(x_0)} h(x)$, and $d_*(x)$ is the distance from $x$ to $\Gamma^*$-the singular boundary of $D$.
\end{lem}

\begin{proof}
By dividing the function $h$ by its supremum norm, we may assume that $h$ is bounded by 1. We shall prove the lemma by contradiction. Suppose the claim fails. Then there exists a sequence of solutions $h_j$ to our problem and
 $x_j \in D\cap B_{1/2}$ with $d_*(x_j) \to 0$, such that
\begin{equation}\label{contra}
  |\nabla h_j(x_j)| \geq j d_* (x^j)^\beta_.
\end{equation}
Now defining $d_j=d_*(x_j)$ and
$$
v_j (x) = \frac{h_j(d_j x + x_j)}{d_j |\nabla h_j(x_j)|}, \text{ in } D_j:=\frac{1}{d_j} (D - x_j),
$$
we see that $v_j$ solves the scaled version of our problem in the scaled domain:
$$
\mathcal{L}_j v_j =0 \text{ in } D_j, \text{ and } |\nabla v_j(0)|=1,
$$
and moreover $v_j$ has the following properties:
$$
0 \leq  v_j (x) =  \frac{h_j(d_j x + x_j)}{d_j |\nabla h_j(x_j)|}\leq \frac{C_0|d_j x + x_j -y_j |^{1+\beta}}{j d_j^{1+\beta} },
$$
where $y_j \in \Gamma^*$ is the closest singular point to $x_j$, and in the second inequality above we have used the estimate in Lemma \ref{barrier}, and estimate (\ref{contra}).
In particular for $|x|<2$ we arrive at
$$
0 \leq  v_j (x) \leq \frac{C  d_j^{1+\beta}}{j d_j^{1+\beta} } \leq \frac{C }{j}   \to   0 \text{ as } j \to \infty.
$$
In other words $v_j$ tends to zero in $D_j \cap B_2$.

Next, and on the other hand, we have by the definition of $v_j$ that $|\nabla v_j (0)|=1$. Also $\partial D_j \cap B_{1/2}$ consists of separated hyperplane or is empty, and therefore $v_j$ will be uniformly $C^{1,\alpha_0}$, for some $\alpha_0>0$ up to the boundary $\partial D_j \cap B_{1/2}$. This would then imply (by elliptic estimates)
$$
1= |\nabla v_j  (0) | \leq C \sup_{D_j \cap B_{2}} v_j \  \to \ 0,
$$
 which is a  contradiction.

\end{proof}

\section{Further Horizon}\label{horison}

In this section we shall discuss some further aspects of the homogenization problem as well as the Fourier approach chosen here, and in previous papers of the authors \cite{ASS}, \cite{ASS2}. Our approach actually works in very general setting, and can be adapted to a regular domains, which are not necessarily convex, but with some control on the vanishing order of the curvature. It should also be noted here that one can not analyze the speed of convergence relying merely on the smoothness of domain without any restriction on the geometry of the boundary. Indeed, as some simple examples show even without singular kernels the integrals of the form $\int\limits_{\Gamma} g( x/ \e) d \sigma(x) $, where $\Gamma$ is a smooth curve, and $g$ is a smooth and 1-periodic function, may converge to its limit with a speed slower than any given rate. This kind of examples are not difficult to construct if one allows the curvature of a surface to be vanishing of infinite order at some point.

Below we shall discuss a few cases that our technique from \cite{ASS}, \cite{ASS2}, and the current paper can be used to derive speed of convergence for the homogenization problem. It should be remarked that the speed deteriorates when the boundary looses convexity or regularity. The departing point for our arguments below will be the setting of this paper, with a second degree divergence type operator, of scalar type. It seems plausible that the ideas can be worked out (with some efforts) for systems, but that would require a better understanding of the behavior of solutions to systems.

In the next few subsections we line up several possible directions, towards which  our results can be generalized. We also suggest some more specific possible approach. Nevertheless, we stress that the reader may see these suggestions as conjecture and not statements or claims of proofs of the ideas.

\subsection{Intersection of finite number of smooth convex domains}
Here we no longer have the smoothness of the domain, and hence the regularity for the Poisson kernel required in \cite{ASS2} does not hold, but one will still have the estimate (\ref{Lem-Poisson-est-with dist}) according to \cite{AL}. In this regard one may try to do a fine covering of the surface to be able to combine our approach from Lemma \ref{Lem-Poisson-on a small strip} to treat the singular parts of the boundary, with the approach from \cite{ASS}, and \cite{ASS2} for smooth boundaries. We believe that this should give some speed of convergence, though worse than the smooth case.

\subsection{General polygonal domains}
The astute reader may have already noticed that the stationary phase analysis part of the paper works out for any polygonal domains, and convexity is not necessary; the Diophantine condition, nevertheless,   is indispensable for our analysis. The convexity was used to hold a good grip on the behavior of the Poisson kernel. For non-convex (or generally Lipschitz) domains we still have (a deteriorated)  control of the Green's function and thus of the Poisson kernel. Indeed, in the estimate (\ref{barrier}) we lose one degree, and
the Green's function in non-convex case becomes $C^\beta$ close to Lipschitz points, with $\beta <1$. Observe that as discussed in the Appendix the most influential parts of the boundary are those that are $(d-2)$-dimensional edges.

It is thus unclear what happens at edges where the Green's function is not as regular as in the convex case. There is a possibility that the pointwise convergence breaks down at such corners (we could not verify this). Since the pointwise convergence takes place at other points
 (as before) one may then conclude $L^p$-convergence locally (away from such points).

\subsection{Not strictly convex domains}
For smooth domains, one may replace the strict convexity requirement of \cite{ASS}, and \cite{ASS2}, by the condition that the principal curvatures do not vanish at certain directions. This should still give some speed of convergence, depending on the number of non-vanishing curvatures, though the speed will be lower than in strict convex case.

\subsection{Local behavior}

A careful inspection of the proof of Theorem \ref{Thm-Pointwise} shows that the pointwise convergence of solutions to (\ref{problem-formulation}) exhibit local behavior. This is due to the fact, that the Poisson kernel has a better regularity at flat boundaries, or near the corners with smaller angle than the worst case, and consequently we will have a better rate for pointwise convergence if we consider $u_\e (x)$ when $x$ is close to these well-behaved boundary points. This in particular indicates that it should be plausible to combine the methods for smooth and strictly convex domains from \cite{ASS}, and those of the current paper, to treat the case of convex $C^2$-regular domains whose boundaries do not have flat portions of positive surface measure with normal vector from $\R \mathbb{Q}^d$ (i.e. rational directions). We remark that for such domains it follows from \cite{LS} that $u_\e$ has a pointwise limit, although the methods of \cite{LS} do not imply effective statements about the convergence rate.

\subsection{Domains with inner boundaries} A further generalization, worthy of mention, are related to domains with disconnected boundary.
One may consider the case of a ring between  two convex domain, or even schlicht-type domains; e.g. the unit  disc minus a line with Diophantine normal direction. Other cases can be a half-plane  with normal of the boundary being Diophantine. All these cases will work perfectly well but will need some more care and work than a few words we use to explain.

\subsection{Other type of oscillation} There are further form of oscillations that can be treated with our method above, and in our earlier works.
One such problem is the oscillation of the source/sink given by
$$
\Delta u_\e (x) = -\mu_\e \qquad \hbox{in } D
$$
with some given boundary data (either oscillating or fixed). Here $\mu_\e := f(x/\e)d\sigma_T$ with $d\sigma_T$ being surfaces
measure over a surface $T \subset D$,  and $f$  a $1$-periodic function. One may still put some restrictions on the geometry of $T$ to allow the Fourier analysis technique above to work. The representation of such a solution through Green's and Poisson kernel can be used along with the arguments in our papers. Observe that the speed for this problem (when the boundary data is fixed) should be faster, due to the fact that Green's function has a less singular behavior than the Poisson kernel. It would also be interesting to investigate the problem for surfaces of codimension greater than one.

\subsection{Other type of  operators}
Generally, any type of problems when an integral representation is available can be treated by this method. This includes for example higher order operators, and equations of divergence type. The above boundary homogenization for parabolic operators is also one further possible direction to be developed.

One may naturally try to generalize the results here to system, but that would require good sources of references (or a carrying-out analysis) for the PDE part of the current paper.


\begin{thebibliography}{99}


      \bibitem{ASS}  Aleksanyan H., Shahgholian H., Sj\"olin P. :  Applications of Fourier analysis in homogenization of Dirichlet problem I. Pointwise estimates. J. Diff. Equations \textbf{254}(6), 2626-2637 (2013)

      \bibitem{ASS2}  Aleksanyan H., Shahgholian H., Sj\"{o}lin P.: Applications of Fourier analysis in homogenization of Dirichlet problem II. $L^p$ estimates. arXiv preprint arXiv:1209.0483 (2012)

 \bibitem{AA99} Allaire G., Amar M.: Boundary layer tails in periodic homogenization.: ESAIM, Control, Optim. Calc. Var. \textbf{4}, 209-243 (1999)

 \bibitem{Anc} Ancona A.: On Positive Harmonic Functions in Cones and Cylinders.  Rev. Mat. Iberoam. \textbf{28}(1) 201-230 (2012)

 \bibitem{AL} Avellaneda M., Lin F.: Homogenization of Elliptic Problems with $L^p$ Boundary Data. Appl. Math. Optim. \textbf{15} 93-107(1987)

% \bibitem{AL-systems} Avellaneda, M., Lin F.: Compactness methods in the theory of homogenization. Comm. Pure Appl. Math., \textbf{40} \textbf{6}, 803--847 (1987)

 \bibitem{BM} Barles, G.,  Mironescu, E.:  On homogenization problems for fully nonlinear equations with oscillating Dirichlet boundary conditions. Asymptotic Analysis, \textbf{82}(3), 187-200 (2013)

\bibitem{BLP} Bensoussan A., Lions J.L., Papanicolaou G.: Asymptotic analysis for periodic structures; Studies in Mathematics and its Applications \textbf{5}; North-Holland 1978

 \bibitem{CK}  Choi S., Kim I.: Homogenization for nonlinear PDEs in general domains with oscillatory Neumann boundary data. arXiv preprint arXiv:1302.5386 (2013)

 \bibitem{CKL}  Choi S., Kim I., K. Lee : Homogenization of Neumann boundary data with fully nonlinear operator. Analysis \& PDE

\bibitem{Feld} W. M. Feldman:  Homogenization of the oscillating Dirichlet boundary condition in general domains. arXiv preprint arXiv:1303.1854 (2013)

\bibitem{GM1} G\'{e}rard-Varet D., Masmoudi N.: Homogenization and boundary layers. Acta Math. \textbf{209} 133-178 (2012)

\bibitem{GM2} G\'{e}rard-Varet D., Masmoudi N.: Homogenization in polygonal domains. J. Eur. Math. Soc. (JEMS) \textbf{13}(5), 1477–1503 (2011)

\bibitem{GT} Gilbarg D., Trudinger Neil: Elliptic partial differential equations of second order. Second edition. Grundlehren der Mathematischen Wissenschaften [Fundamental Principles of Mathematical Sciences], 224. Springer-Verlag, Berlin, 1983. xiii+513 pp.

 \bibitem{Otto-Gloria} Gloria A., Otto F.: An optimal variance estimate in stochastic homogenization of discrete elliptic equations. The Annals of Probability 39.3, 779-856 (2011)

    \bibitem{GW} Gr\"{u}ter M., Widman K.-O.: The Green function for uniformly elliptic equations. Manuscripta Math. \textbf{37}, 303-342 (1982)

\bibitem{DoMu} Dolzmann G., M\"{u}ller S.: Estimates for the Green's matricies of elliptic systems by $L^p$ theory. Manuscripta Math. \textbf{88}, 261-273 (1995)

%    \bibitem {KLS} Kenig C., Lin F., Shen Zh.: Periodic Homogenization of Green and Neumann Functions. arXiv preprint arXiv:1201.1440v1 (2012)

      \bibitem{LS} Lee K-A., Shahgholian H.: Homogenization of the boundary value for the Dirichlet problem; arXiv preprint arXiv:1201.6683 (2012)

      \bibitem{LY} Lee K-A.,  Yoo M. : Homogenization of fully nonlinear elliptic equations with oscillating dirichlet boundary data.  arXiv preprint arXiv:1304.7070 (2013)

     \bibitem{K} Kuran \"U. : On NTA-conical domains:  J. London Math. Soc. \textbf{40}(2), 467-475 (1989)

    \end{thebibliography}
\end{document}